\documentclass[a4paper]{amsart}
\usepackage{amsmath,amsthm,amssymb}
\usepackage{amscd}
\usepackage[dvipdfmx,hiresbb]{graphicx}
\usepackage{mathrsfs}
\usepackage[all]{xy}
\usepackage{graphicx}
\usepackage{color}

\newtheorem{thm}{Theorem}[subsection]
\newtheorem{lem}[thm]{Lemma}
\newtheorem{cor}[thm]{Corollary}
\newtheorem{prop}[thm]{Proposition}

\theoremstyle{definition}

\newtheorem{defn}[thm]{Definition}

\newtheorem{Rem}[thm]{Remark}

\theoremstyle{remark}

\numberwithin{equation}{subsection}

\def\XXint#1#2#3{{\setbox0=\hbox{$#1{#2#3}{\int}$}
\vcenter{\hbox{$#2#3$}}\kern-.5\wd0}}

\newcommand{\R}{\mathbb{R}}

\newcommand{\Cpl}{\mathsf{Cpl}}

\newcommand{\supp}{\mathrm{supp}\,}

\newcommand{\bary}{\mathsf{b}}
\newcommand{\Opt}{\mathsf{Opt}}
\newcommand{\Var}{\mathsf{Var}}

\newcommand{\calC}{\mathcal{C}}
\newcommand{\calD}{\mathcal{D}}

\newcommand{\calH}{\mathcal{H}}

\newcommand{\calK}{\mathcal{K}}
\newcommand{\calL}{\mathcal{L}}

\newcommand{\calP}{\mathcal{P}}

\newcommand{\vol}{\rm vol}

\newcommand{\bbC}{\mathbb{C}}

\newcommand{\bbN}{\mathbb{N}}

\newcommand{\bbR}{\mathbb{R}}

\newcommand{\bbZ}{\mathbb{Z}}

\newcommand{\frakh}{\mathfrak{h}}

\newcommand{\scrP}{\mathscr{P}}

\DeclareMathOperator{\Diam}{Diam}

\usepackage[abbrev,nobysame]{amsrefs}

\makeatletter
\def\@makefnmark{%
\leavevmode
\raise.9ex\hbox{\check@mathfonts
\fontsize\sf@size\z@\normalfont%
\@thefnmark}%
}
\makeatother

\title{Distance functions on convex bodies and symplectic toric manifolds}
\author{H.~Fujita, Y.~Kitabeppu, A.~Mitsuishi}
\address[Hajime Fujita]{Japan Women's University}
\email{fujitah@fc.jwu.ac.jp}
\address[Yu Kitabeppu]{Kumamoto University}
\email{ybeppu@kumamoto-u.ac.jp}
\address[Ayato Mitsuishi]{Fukuoka University}
\email{mitsuishi@fukuoka-u.ac.jp}


\begin{document}

\maketitle
 \begin{abstract}
 In this paper we discuss three distance functions on the set of convex bodies. 
 In particular we study the convergence of Delzant polytopes, which are fundamental objects in 
 symplectic toric geometry. 
 By using these observations, we derive some convergence theorems for symplectic toric manifolds 
 with respect to the Gromov-Hausdorff distance. 
 \end{abstract}

\tableofcontents
 %
 %
\section{Introduction}
Convex polytopes, or more generally convex bodies, are classical and important objects in geometry. 
There are many results in which structures or properties of convex polytopes are shown to have deep connections
with other objects, through algebraic or combinatorial procedures. 
Among other such results, there is the {\it Delzant construction}~\cite{Delzant}, which is well known in symplectic geometry. 
Using the Delzant construction one obtains a natural bijective correspondence 
between the set of {\it Delzant polytopes} and the set of {\it symplectic toric manifolds}. 
 Under this correspondence, the geometric data of symplectic toric manifolds 
are encoded as combinatorial or topological properties of their corresponding polytopes. 
For example, the cohomology ring of symplectic toric manifolds can be recovered completely as the  
{\it Stanley-Reisner ring} of the associated polytope. 
See e.g. \cite{BuchstaberPanov} for more details on this dictionary between Delzant polytopes and symplectic toric manifolds.

The purpose of our project is to further develop aspects of this kind of correspondence from the viewpoint of 
Riemannian or metric geometry. The present paper contains two parts.  
Firstly, we establish relationships between three natural distance functions on the set of convex bodies. 
The first function $d^W$ is defined by the {\it Wasserstein distance} of probability measures associated with convex bodies. 
The Wasserstein distance is a quite important tool in recent developments of geometric analysis for metric measure spaces. 
The second distance $d^V$ is defined by the Lebesgue volume of the symmetric difference of convex bodies.
This distance function is natural from the viewpoint of symplectic geometry and is studied in \cite{PPRS} and \cite{FujitaOhashi}. 
The third function $d^H$ is the Hausdorff distance, which is a classical and basic tool in geometry of convex bodies. 
The main result of the first part of this paper is as follows. 

\medskip

\noindent
{\bf Theorem 1 (Theorem~\ref{thm:equivdVdWdH}). } 
{\it The metric topologies determined by the distance functions $d^W$, $d^V$ and $d^H$ coincide with each other. }

\medskip

Secondly, we investigate the relationship between the 
metric geometry of Delzant polytopes and the Riemannian geometry of symplectic toric manifolds through the Delzant construction.  
Here we equip each symplectic toric manifold with a K\"ahler metric called the {\it Guillemin metric}~\cite{Guillemin}, 
and we regard a symplectic toric manifold as a Riemannian manifold. 
The main results in the second part of this paper are the following. 

\medskip

\noindent
{\bf Theorem~2 (Theorem~\ref{thm:convGH}). } 
{\it For a sequence of Delzant polytopes $\{P_i\}_i$ in $\R^n$, suppose that 
$\{P_i\}_i$ converges to a Delzant polytope $P$ in $\R^n$ in the $d^H$-topology (hence also in the $d^W$-topology and $d^V$-topology),  
and the limit of the numbers of facets of $\{P_i\}_i$ coincides with that of $P$. Then 
the sequence of symplectic toric manifolds $\{M_{P_i}\}_i$ with the Guillemin metric converges to $M_P$  in 
the torus-equivariant Gromov-Hausdorff topology. }

\medskip

As a corollary (Corollary~\ref{cor:stability}), we also have a torus-equivariant stability theorem in the setting of converging symplectic toric manifolds. 

\medskip

\noindent
{\bf Theorem~3 (Theorem~\ref{thm:GHHvertex}, Theorem~\ref{thm:GHH}). } 
{\it For a sequence of Delzant polytopes $\{P_i\}_i$ in $\R^n$ and a Delzant polytope $P$ in $\R^n$, 
suppose that the corresponding sequence of symplectic toric manifolds $\{M_{P_i}\}_i$ converges to $M_P$ 
in the torus-equivariant measured Gromov-Hausdorff topology.
Then we have : 
\begin{itemize}
\item the fixed point set of $M_{P_i}$ converges to that of $M_P$. In particular we have the lower semi-continuity of the Euler characteristic, and  
\item we have a sequence which converges to $P$ in $d^H$-topology by using $\{P_i\}_i$ and the approximation maps for $\{M_{P_i}\}_i$. 
\end{itemize} 
}

\medskip

We emphasize that there are no hypotheses on the curvature in the statement of the above theorem. 
By incorporating \lq\lq potential functions\rq\rq  as in \cite{Aberu} we may treat more general 
torus-invariant Riemannian metrics of symplectic toric manifolds which are not necessarily Guillemin metrics.

In the present paper, we only consider the non-collapsing case. 
It is surely interesting to attack the same problems under collapsing limit, and 
we will discuss this in a subsequent paper. 
In addition, our general setting of convex bodies in the first part of this paper is motivated by the fact that non-Delzant polytopes are increasingly important in the context of toric degenerations 
of integrable systems or projective varieties as in \cite{HaradaKaveh}, \cite{NishinouNoharaUeda} and so on. 

This paper is organized as follows. 
In Section~\ref{Three distance functions} we introduce three distance functions on the set of convex bodies. 
In Section~\ref{Relation of distance fuctions} we show that the three corresponding metric topologies coincide. 
Note that the equivalence between the distance function defined by the volume and the Hausdorff 
distance is classically known, by \cite{ShephardWebster} for example. 
In \cite{PPRS} Pelayo-Pires-Ratiu-Sabatini studied several properties of the moduli space of Delzant polytopes with respect to the natural action of integral affine transformations. 
This moduli space arises naturally from an equivalence relation of symplectic toric manifolds 
known as weak equivalence, and we expect it to be an important object in a subsequent research.
We also give comments on the distance function and the associated topology on this moduli space which were studied in \cite{FujitaOhashi}.
In Section~\ref{Delzant polytopes} we discuss the definition of Delzant polytopes and 
the description of Guillemin metric on the corresponding symplectic toric manifolds. 
In Section~\ref{Convergence of polytopes} we discuss the relation between the convergence of Delzant polytopes and 
the convergence of symplectic toric manifolds. 
In Appendix~\ref{Preliminaries} we record several facts on probability measures and Wasserstein distance. 
In Appendix~\ref{Disintegration} we provide a disintegration theorem which 
is important in the proof of Theorem~\ref{thm:GHH}. 

\medskip

\noindent
{\bf Acknowledgement.} 
This work was partially done while the first author was visiting the Department of Mathematics, University of Toronto, and 
the Department of Mathematics and Statistics, McMaster University. 
He would like to thank both institutions for their hospitality, especially for M.~Harada. 
He is also grateful to Y.~Karshon and X.~Tang for fruitful discussions. 
The first author is partly supported by Grant-in-Aid for Scientific Research (C) 18K03288. 
The second author is partly supported by Grant-in-Aid for Early-Career Scientists 18K13412. 
The third author is partly supported by Grant-in-Aid for Young Scientists (B) 15K17529 and Scientific Research (A) 17H01091. 
Finally, the authors would like to express gratitude to K.~Ohashi who gave us a chance to begin this research. 

\medskip

\noindent
{\bf Notations.} 
For a metric space $(X,d)$, a subset $Y$ of $X$, a point $x$ in $X$ and a positive real number $r$ we use the following notations. 
\begin{itemize}
\item $B(x, r):= \{y\in X \ | \ d(x,y)<r\}$ : open ball of radius $r$ centered at $x$. 
\item $B(Y, r):= \left\{y\in X \ \middle| \ \displaystyle\inf_{y'\in Y}d(y,y')<r\right\}$ : open $r$-neighborhood of $Y$. 
\item ${\rm dist}(x,A):=\displaystyle\inf\{d(x,y) \ | \ y\in A\}$ : distance between $x$ and $A$. 
\item $\Diam(A):=\displaystyle\sup\{d(y,y') \ | \ y,y'\in A\}$ : diameter of $A$. 
\end{itemize}
We use the notation $\| \cdot\|$ (resp. $\langle\cdot, \cdot\rangle$) for the Euclidean norm (resp. inner product) 
on the Euclidean spaces.  We also use the notation $|A|$ for the Lebesgue measure of a Lebesgue measurable subset $A$. 

\section{Three distance functions on the set of convex bodies}
\label{Three distance functions}
Let $\calC_n$ be the set of all convex bodies in $\bbR^n$, i.e., $\calC_n$ is the set of all bounded closed convex sets  obtained as closures of open subsets in $\R^n$.

\subsection{$L^2$-Wasserstein distance}

For each $C\in\calC_n$ let $m_C$ be the probability measure on $\bbR^n$ with compact support defined by 
\[
m_C:=\frac{\chi_C}{\calH^n(C)}\calH^n, 
\]
where $\chi_C$ is the characteristic function of $C$ and $\calH^n$ is the $n$-dimensional Hausdorff measure on $\R^n$. Of course $\calH^n$ is equal to the $n$-dimensional Lebesgue measure $\calL^n$, however, since we put on the field of view of collapsing phenomena of convex bodies into lower dimensional objects,  we prefer to use the Hausdorff measure. 

\begin{defn}
Define a function $d^W:\calC_n\times \calC_n\to \bbR_{\geq 0}$ by 
\[
d^W(C_1, C_2):=W_2(m_{C_1}, m_{C_2}), 
\]where $W_2$ is the $L^2$-Wasserstein distance on the set of all probability measures on $\R^n$ with finite quadratic moment. 
\end{defn}
See Appendix~\ref{App.A} for basic definitions and facts on $L^2$-Wasserstein distance. 

\begin{lem}
$d^W$ is a distance function on $\calC_n$. 
\end{lem}
\begin{proof}
Symmetricity, triangle inequality and non-negativity are clear. The non-degeneracy follows from 
the equivalence between the conditions 
$d^W(C_1, C_2)=W_2(m_{C_1}, m_{C_2})=0$ and 
$C_1=\supp(m_{C_1})=\supp(m_{C_2})=C_2$.
\end{proof}

\subsection{Lebesgue volume}
For $C_1, C_2\in\calC_n$, let $d^V(C_1, C_2)$ be the Lebesgue volume of the symmetric difference $C_1\bigtriangleup C_2:=(C_1\setminus C_2)\cup (C_2\setminus C_1)$ : 
\[
d^V(C_1, C_2):=|C_1\bigtriangleup C_2|=\int_{\R^n}\chi_{C_1 \bigtriangleup C_2}(x)\calL^n(dx). 
\]
This $d^V$ is indeed a distance function on $\calC_n$ and used in a study of convex bodies classically. 
See \cite{Dinghas} or \cite{ShephardWebster} for example. 

\subsection{Hausdorff distance}
Let $d^H$ be the Hausdorff distance on the set of all compact subsets in $\R^n$.We also denote the restriction of $d^H$ to $\calC_n$ by the same letter $d^H$ : 
\[
d^H(C_1, C_2):=\max\{\max_{x\in C_1}\min_{y\in C_2}\|x-y\|,  \ \max_{y\in C_2}\min_{x\in C_1}\|x-y\| \} \quad 
(C_1, C_2\in \calC_n). 
\]

\section{Relation of distance functions}
\label{Relation of distance fuctions}
\subsection{Equivalence among $d^W$, $d^V$ and $d^H$}

It is known that two distance functions $d^V$ and $d^H$ give the same metric topology. 
More precisely in \cite{ShephardWebster} it is shown that a sequence $\{P_i\}_i$ in $\calC_n$ converges to $Q\in\calC_n$ in $d^V$ if and only if 
it converges to $Q$ in $d^H$. 

\begin{lem}\label{dVtodW}
For a sequence $\{P_i\}_i$ in $\calC_n$ and $Q\in\calC_n$, if $d^V(P_i,Q)\to 0 \ (i\to\infty)$, then we have $d^W(P_i ,Q)\to 0 \ (i\to \infty)$. 
\end{lem}

\begin{proof}
Since $\displaystyle\lim_{i\to \infty}d^V(P_i, Q)=0$ implies $\displaystyle\lim_{i\to \infty}d^H(P_i, Q)=0$
we may assume that 
\[
K_i:=\Diam(P_i)\leq 100K:=100\Diam(Q),  
\]
and $|\log(|P_i|/|Q|)|<\epsilon$ for small $\epsilon>0$ and any $i$ large enough.  
Now we define couplings $\xi_i\in{\Cpl}(m_{P_i}, m_Q)$ ($i=1,2$) by 
\[
\xi_1(X_1\times X_2):=m_Q(X_1\cap X_2\cap P_i\cap Q)m_{P_i}(X_2)+
m_Q(X_1\setminus (X_2\cap P_i))m_{P_i}(X_2) 
\]when $|Q|\geq|P_i|$ and 
\[
\xi_2(X_1\times X_2):=m_Q(X_1)m_{P_i}(X_1\cap X_2\cap P_i\cap Q)+
m_Q(X_1) m_{P_i}(X_2\setminus(Q\cap X_1))
\]when $|P_i|\geq |Q|$. 
Then we have 

\begin{align}
d^W(P_i,Q) \notag &\leq 
\sqrt{\int_{\R^n\times \R^n}\Vert x-y \Vert^2\xi_1(dx,dy)}+
\sqrt{\int_{\R^n\times \R^n}\Vert x-y \Vert^2\xi_2(dx,dy)} \notag \\ 
  &\leq \sqrt{\frac{\vert Q\setminus P_i\vert}{\vert Q\vert}\cdot (101 K)^2}+\sqrt{\frac{\vert P_i\setminus Q\vert}{\vert P_i\vert}\cdot (101K)^2}  \notag\\
  &\leq 2\cdot 101 K\sqrt{\frac{\vert Q\bigtriangleup P_i\vert}{\min\{\vert Q\vert,\vert P_i\vert\}}}\notag\\
  &\leq 2\cdot 101 K\sqrt{\frac{d^V(Q,P_i)}{e^{-\epsilon}\vert Q\vert}}
\rightarrow0 \ ({\rm as} \ i\to\infty).  \notag
 \end{align}

\end{proof}

\begin{lem}\label{dWtodV}
For a sequence $\{P_i\}_i$ in $\calC_n$ and $Q\in\calC_n$, if $d^W(P_i,Q)\to 0 \ (i\to\infty)$, then we have $d^V(P_i ,Q)\to 0 \ (i\to \infty)$. 
\end{lem}

\begin{proof}
Suppose that $d^W(P_i,Q)\to 0 \ (i\to\infty)$. 
Then, $m_i:=m_{P_i}$ converges weakly to $m:=m_Q$, in particular, we have
\[
m_i(Q)=\frac{|P_i\cap Q|}{|P_i|}\to m(Q)=1 
\]
by Theorem~\ref{thm:propertyweak}. 
Since $|P_i\cap Q|\leq |Q|$ we have $|P_i|$ is bounded, and hence, 
\[
\frac{|P_i|}{|Q|} <c
\] for some $c>0$. 
Corollary~\ref{cor:inprob} implies that  for two probability measures $m_i$ and $m$ 
there exist a sequence of Borel measurable maps $\{T_i:\R^n\to\R^n\}_i$ such that 
$(\mathtt{id}\times T_i)_*m\in\Opt(m, m_i)$ for all $i$ and 
\[
m(\{x\in Q \  | \ \|x-T_i(x)\|\geq a\})=m(\{x\in \R^n \ | \ \|x-T_i(x)\|\geq a\}) \to 0 \ (i\to\infty)
\]
for all $a>0$. 
Let us fix an arbitrary positive number $\epsilon$ and set 
\[
\xi:= \frac{\epsilon}{(c+1)(|Q|+1)}. 
\]
Choose $\eta$ small enough so that 
\[
|B(Q,\eta)\setminus Q|<
\xi. 
\]
There exists $N\in\bbN$ such that 
\[
m(\{x\in Q \ | \ \|T_i(x)-x\|\geq \eta\})<\xi
\]
for all $i \geq N$. 
Take and fix $i> N$. 
For $x\in Q$ we put $r_x^i:=\|x-T_i(x)\|$. Then we have 
$\displaystyle Q\subset \bigcup_{x\in Q}B(x, r_x^i)$. 
We put 
\[
U^i:=\bigcup_{x\in Q, r_x^i\leq \eta}\overline{B(x, r_x^i)}. 
\]We have 
\begin{align}
|U^i\setminus Q|&\leq |B(Q,\eta)\setminus Q|<\xi, \notag \\ 
|Q\setminus U^i|&=|Q| m(Q\setminus U^i)\notag \\ 
&\leq |Q| m(\{x\in Q \ | \  \|x-T_i(x)\|)\geq\eta\}) \notag \\ 
&<|Q|\xi \notag, 
\end{align}
and hence, $|Q\bigtriangleup U^i|<(|Q|+1)\xi$. 
On the other hand we have 
\begin{align}
|P_i\setminus U^i|&=|P_i|m_i(P_i\setminus U^i) \notag \\
&= |P_i|(T_i)_*m(P_i\setminus U^i) \notag \\ 
&= |P_i| m(T_i^{-1}(P_i)\setminus T_i^{-1}(U^i)).  \notag 
\end{align}
Since $(T_i)_*m=m_i$ we have that $T_i^{-1}(P_i)=Q$ ($m$-a.e.). 
This fact and $T_i^{-1}(\overline{B(x,r_x^i)})\ni x$ imply that \[
T_i^{-1}(U^i)\supset \{x\in Q \ | \ \|x-T_i(x)\|\leq \eta\}. 
\] 
In particular we have 
\[
|P_i\setminus U^i|\leq |P_i| m(\{x\in Q \ | \ \|x-T_i(x)\|>\eta\})\leq |P_i|\xi. 
\]
Similarly we have 
\begin{align}
\vert U^i\setminus P_i\vert&=\vert P_i\vert m_i(U^i\setminus P_i)=\vert P_i\vert m(T_i^{-1}(U^i)\setminus Q)\notag\\
&\leq\vert P_i\vert m(B(Q,\eta)\setminus Q)=\frac{\vert P_i\vert}{\vert Q\vert}\vert B(Q,\eta)\setminus Q\vert\notag\\
&<\frac{\vert P_i\vert}{\vert Q\vert}\xi\leq c\xi, \notag
\end{align}
and hence $\vert U^i\bigtriangleup P_i\vert\leq (\vert P_i\vert+c)\xi$. 
Therefore we have 
\begin{align}
&d^V(P_i,Q)=\vert Q\bigtriangleup P_i\vert\leq \vert Q\bigtriangleup U^i\vert+\vert U^i\bigtriangleup P_i\vert\notag\\
&\leq (\vert Q\vert+\vert P_i\vert+c+1)\xi\leq ((1+c)\vert Q\vert+c+1)\xi=(1+c)(\vert Q\vert+1)\xi\notag\\
&=\epsilon. 
\notag
\end{align}
Since $\epsilon > 0$ is arbitrary, we obtain the conclusion, that is, $d^V(P_i,Q)\rightarrow 0$.
\end{proof}

As a corollary of Lemma~\ref{dVtodW} and Lemma~\ref{dWtodV} we have the following by 
 Kratowski's axiom and the coincidence between the metric topology of $d^V$ and $d^H$ as 
 shown in \cite{ShephardWebster}. 

\begin{thm}\label{thm:equivdVdWdH}
Three metric topologies on $\calC_n$ determined by $d^W$, $d^V$ and $d^H$ coincide with each other. 
\end{thm}


\subsection{Moduli space of convex bodies and its topology}
We introduce the moduli space of convex bodies following \cite{FujitaOhashi} and \cite{PPRS}. 
Let $G_n:={\rm AGL}(n,\bbZ)$ be the integral affine transformation group. Namely $G_n$ is the direct product ${\rm GL}(n,\bbZ)\times \bbR^n$ as a set and the multiplication on $G_n$ is defined by 
\[
(A_1, t_1)\cdot(A_2, t_2):=(A_1A_2, A_1t_2+t_1)
\]for each $(A_1, t_1), (A_2, t_2)\in G_n$. 
This group $G_n$ acts on $\calC_n$ in a natural way, and $C\in \calC_n$ and $C'\in\calC_n$ are called {\it $G_n$-congruent} if $C$ and $C'$ are contained in the same $G_n$-orbit. 
\begin{defn}
The moduli space of convex bodies $\widetilde\calC_n$ with respect to the $G_n$-congruence is defined by the quotient 
\[
\widetilde\calC_n:=\calC_n/G_n. 
\]
\end{defn}
Let $\pi $ be the natural projection from $\calC_n$ to $\widetilde\calC_n$. 
\begin{defn}\label{deftilded}
Define a function $D^V:\widetilde\calC_n\times\widetilde\calC_n\to \R$ by 
\[
D^V(\alpha, \beta):=\inf\{d^V(P_1, P_2) \ | \ \pi(P_1)=\alpha, \pi(P_2)=\beta\}
\] for $(\alpha, \beta)\in \widetilde\calC_n\times\widetilde\calC_n$. 
\end{defn}

\begin{thm}[\cite{FujitaOhashi}]
$D^V$ is a distance function on $\widetilde\calC_n$ and its metric topology coincides with the quotient topology 
induced from $\pi$. 
\end{thm}

This $G_n$-action and the moduli space $\widetilde\calC_n$ arise naturally in the context of the geometry of 
symplectic toric manifolds. In the subsequent sections we will discuss from such point of view. 

\begin{Rem}\label{rem:minvar}
As it is noted in \cite{FujitaOhashi} we can not define a distance function on $\widetilde\calC_n$ 
by using the infimum of $d^H$ (or $d^W$) among all representatives, though, one may hope that 
by considering infimum of $d^H$ among only \lq\lq standard\rq\rq   representatives  
we can define a distance function on $\widetilde\calC_n$. 
One possible candidates of \lq\lq standard\rq\rq   representatives are the minimum 
variance (or quadratic moment) elements in the following sense. 

For each $C\in \calC_n$ define its variance by 
\[
\Var(C):=\frac{1}{|C|}\int_C\|x-\bary(C)\|^2\calL^n(dx), 
\]where $\bary(C)$ is the barycenter of $C$ which is determined uniquely by the condition 
\[
\langle \bary(C), y \rangle =\int_{\bbR^n}\langle x, y \rangle \calL^n(dx)
\] for any $y\in\bbR^n$. See \cite{Stbary} for example. 
The minimum variance element $C\in\calC_n$ is an element of 
\[
\mathsf{argmin}\left\{ \Var(C') \ | \ C'\in\calC_n \  {\rm is} \ G_n \text{-congruent  to} \ C\right\}. 
\] One can see that for any $C\in\calC_n$ there exist at least one and finitely many minimum variance elements 
which have the common barycenter are $G_n$-congruent to $C$. 
\end{Rem}


\section{Delzant polytopes and symplectic toric manifolds }
\label{Delzant polytopes}
\subsection{Delzant polytopes, symplectic toric manifolds and their moduli space}

\begin{defn}\label{def:delzantpoly}
A convex polytope $P$ in $\R^n$ is called a {\it Delzant polytope} if $P$ satisfies the following conditions : 
\begin{itemize}
\item $P$ is simple, that is, each vertex of $P$ has exactly $n$ edges. 
\item $P$ is rational, that is, at each vertex all directional vectors of edges can be taken as integral vectors in $\bbZ^n$. 
\item $P$ is smooth, that is, at each vertex we can take integral directional vectors of edges as a $\bbZ$-basis of $\bbZ^n$ in $\bbR^n$. 
\end{itemize}
We denote the subset of $\calC_n$ consisting of all Delzant polytopes by $\calD_n$ and define their moduli space by $\widetilde\calD_n:=\calD_n/G_n$. 
\end{defn}
Recall that a {\it symplectic toric manifold} $(M,\omega, \rho, \mu)$ is a data consisting of 
\begin{itemize}
\item a compact connected symplectic manifold $(M,\omega)$ of dimension $2n$, 
\item a homomorphism $\rho$ from the $n$-dimensional torus $T^n$ to the group of symplectomorphisms of $M$ which gives a Hamiltonian action of $T^n$ on $M$ and 
\item a moment map $\mu:M\to \bbR^n=({\rm Lie}(T^n))^*$. 
\end{itemize}

The famous Delzant construction gives a correspondence between Delzant polytopes  and symplectic toric manifolds. 
\begin{thm}[\cite{KarshonKessler}]
The Delzant construction gives a bijective correspondence between $\widetilde\calD_n$ and the set of all weak isomorphism classes of $2n$-dimensional symplectic toric manifolds. 
\end{thm}

Here two symplectic toric manifolds $(M_1, \omega_1, \rho_1,\mu_1)$ and $(M_2, \omega_2, \rho_2, \mu_2)$ are {\it weakly isomorphic}\footnote{In \cite{KarshonKessler} the equivalence relation \lq\lq weakly isomorphism \rq\rq \ is called 
just \lq\lq  equivalent \rq\rq. In this paper we follow the terminology in \cite{PPRS}.} if there exist a diffeomorphism $f:M_1\to M_2$ and a group isomorphism $\phi:T^n\to T^n$ such that 
\[
f^*\omega_2=\omega_1 \ {\rm and} \ \rho_1(g)(x)=\rho_2(\phi(g))(f(x)) \ {\rm for\ all }  \ (g,x)\in T^n\times M_1. 
\]

Based on the above fact the moduli space $\widetilde\calD_n$ is also called the {\it moduli space of toric manifolds} in \cite{PPRS}. 
In \cite{PPRS} they show that $(\calD_n,d^V)$ is neither complete nor locally compact and 
$\widetilde\calD_2$ is path connected.

\subsection{Brief review on the Delzant construction}\label{subsec:Delcon}
For later convenience we give a brief review on the Delzant construction here. 

Let $P$ be an $n$-dimensional Delzant polytope and
\begin{equation}\label{l^r}
l^{(r)}(x):=\langle x, \nu^{(r)}\rangle - \lambda^{(r)}=0 \quad (r=1,\cdots, N)
\end{equation}
a system of defining affine equations on $\bbR^n$ of facets of $P$, 
each $\nu^{(r)}$ being inward pointing normal vector of $r$-th facet and $N$ is the number of facets of $P$. In other words $P$ can be described as 
\[
P=\bigcap_{r=1}^N\{x\in\bbR^n \ | \ l^{(r)}(x)\geq 0\}. 
\]
 We may assume that each $\nu^{(r)}$ is primitive\footnote{An integral vector $u$ in $\bbR^n$ is called {\it primitive} if $u$ cannot be described as $u=ku'$ for another integral vector $u'$ and $k\in\bbZ$ with $|k|>1$. } and they form a $\bbZ$-basis of $\bbZ^n$. 
Consider the standard Hamiltonian action of the $N$-dimensional torus $T^N$ on $\bbC^N$ with the moment map  
\[
\tilde\mu:\bbC^N\to (\bbR^N)^*={\rm Lie}(T^N)^*, \ (z_1, \ldots, z_N)\mapsto  -(|z_1|^2, \ldots, |z_N|^2)+(\lambda^{(1)}, \ldots, \lambda^{(N)}). 
\]
Let $\tilde\pi:\bbR^N\to \bbR^n$ be the linear map defined by $e_r\mapsto\nu^{(r)}$, where $e_r$ ($r=1,\ldots, N$) is the 
$r$-th standard basis of $\bbR^N$. 
Note that $\tilde\pi$ induces a surjection $\tilde\pi:\bbZ^N\to\bbZ^n$ between the standard lattices by the 
last condition in Definition~\ref{def:delzantpoly}, and hence it induces surjective homomorphism between tori,  still denoted by $\tilde\pi$, 
\[
\tilde\pi:T^N=\bbR^N/\bbZ^N\to T^n=\bbR^n/\bbZ^n. 
\]
Let $H$ be the kernel of $\tilde\pi$ which is an $(N-n)$-dimensional subtorus of $T^N$ and 
$\frakh$ its Lie algebra.  
We have exact sequences 
\[
1\to  H\stackrel{\iota}{\to} T^N\stackrel{\tilde\pi}{\to} T^n\to 1,
\]
\[
0 \to  \frakh\stackrel{\iota}{\to} \bbR^N\stackrel{\tilde\pi}{\to} \bbR^n\to 0
\]and its dual 
\[
0 \to (\bbR^n)^* \stackrel{\tilde\pi^*}{\to} (\bbR^N)^*\stackrel{\iota^*}{\to}  \frakh^*\to 0, 
\]where $\iota$ is the inclusion map. 
Then the composition $\iota^*\circ\tilde\mu:\bbC^N\to \frakh^*$ is the associated moment map 
of the action of $H$ on $\bbC^N$. 
It is known that $(\iota^*\circ\tilde\mu)^{-1}(0)$ is a compact submanifold of $\bbC^N$ and 
$H$ acts freely on it.  
We obtain the desired symplectic manifold $M_P:=(\iota^*\circ\tilde\mu)^{-1}(0)/H$ equipped with 
a natural Hamiltonian $T^N/H=T^n$-action. 
Note that the standard flat K\"ahler structure on $\bbC^N$ induces a K\"ahler structure on $M_P$. 
The associated Riemannian metric is called the {\it Guillemin metric}. 

\begin{Rem}
In the above set-up we assume that the number of facets of $P$, say $N$,  is equal to that of the defining inequalities,  
though, it is possible to consider the similar construction formally for any system of inequalities 
which has more than $N$ inequalities. 
Such a construction may produce a symplectic toric manifold 
equipped with metric which is not isometric to the Guillemin metric. 
\end{Rem}

There exists an explicit description of the Guillemin metric.  
We give the description following \cite{Aberu}. 
Consider a smooth function 
\begin{equation}\label{g_P}
g_P:=\frac{1}{2}\sum_{r=1}^Nl^{(r)}\log l^{(r)} : P^{\circ}\to \bbR, 
\end{equation}
where $P^{\circ}$ is the interior of $P$. 
It is known that $M_P^{\circ}:=\mu_P^{-1}(P^{\circ})$  is an open dense subset of $M_P$ on which $T^n$ acts freely and there exists a diffeomorphism $M_P^{\circ}\cong P^{\circ}\times T^n$. 
Under this identification $\omega_P|_{M_P^{\circ}}$ can be described as 
\[
\omega_P|_{M_P^{\circ}}=dx\wedge dy=\sum_{i=1}^ndx_i\wedge dy_i
\]using the standard coordinate\footnote{Here we regard $T=T^n=(S^1)^n$ and $S^1=\bbR/\bbZ$. } $(x,y)=(x_1, \ldots, x_n, y_1, \ldots, y_n)\in P^{\circ}\times T^n$. The coordinate on $M_P^{\circ}$ induced from $(x,y)\in P^{\circ}\times T^n$ is called the {\it symplectic coordinate} on $M_{P}$. 

\begin{thm}[\cite{Guillemin}]\label{Guillemins metric}
Under the symplectic coordinates $(x,y)\in P^{\circ}\times T^n\cong M_P^{\circ}\subset M_P$, the Guillemin metric can be described as 
\[
\begin{pmatrix}
G_P & 0 \\
0 & G_P^{-1}
\end{pmatrix}, 
\]where $\displaystyle G_P:={\rm Hess}_x(g_P)=\left(\frac{\partial^2 g_P}{\partial x_k\partial x_l}\right)_{k,l=1,\ldots, n}$ is the Hessian of $g_P$. 
\end{thm}

\begin{Rem}\label{rem:weakisom->isometric}
If $P$ and $P'$ in $\calD_n$ are $G_n$-congruent, 
then the corresponding Riemannian manifolds $M_P$ and $M_{P'}$ are isometric to each other. 
In fact as it is noted in \cite[Section~3.3]{Aberu}, for $\varphi\in G_n$ we have an isomorphism between 
$M_P$ and $M_{\varphi(P)}$ as K\"ahler manifolds. The isomorphism is induced by the 
map $P\times T\to \varphi(P)\times T$, $(x,t)\mapsto (\varphi(x), ((\varphi_*)^{-1})^{T}(t))$, where 
$(\ )^{T}$ is the transpose and 
$\varphi_*$ is the automorphism of $T$ which is induced by $\varphi$. 
\end{Rem}

\section{Convergence of polytopes and symplectic toric manifolds}
\label{Convergence of polytopes}
Hereafter we do not often distinguish a sequence itself and a subsequence of it. 
\subsection{Convergence of polytopes and related quantities}
For a polytope $P$ in $\bbR^n$ we denote the set of all $k$-dimensional faces of $P$ by $\{F_{k}^{(r)}(P)\}_r$. In particular we denote the set of all facets by $\{F^{(r)}(P)\}_r$.  
We often omit the superscript $r$ for simplicity and denote each face by $F_k(P)$ for example. 

\begin{prop}\label{FtoF}
For a sequence $\{P_i\}_i\subset \calD_n$ suppose that $d^H(P_i, P)\to 0 \ (i\to \infty)$ for $P\in\calD_n$. Then for any $x\in F(P)$ there exists a sequence  $\{x_i\in F(P_i)\}_i$ such that $x_i\to x \ (i\to\infty)$. 
\end{prop}
\begin{proof}
For $x\in F(P)$ suppose that 
\[
\limsup_{i\to\infty}{\rm dist}(x, \partial P_i)>10\epsilon
\]for some $\epsilon>0$. 
We may assume that 
\[
B(x, 9\epsilon)\cap \partial P_i =\emptyset
\]for any $i$ by taking a subsequence. 
By the above assumption and $P_i\to P$ in $d^H$ there exists a sequence $\{y_i\in P_i^{\circ}\}_i$ such that $y_i\to x$. For any $i$ large enough we may assume that $\|x-y_i\|<\epsilon$. Then we have 
\[
{\rm dist}(y_i, \partial P_i)\geq {\rm dist}(x, \partial P_i)-\|x-y_i\|\geq8\epsilon, 
\]and hence, 
\[
B(y_i,8\epsilon)\cap \partial P_i=\emptyset. 
\]In particular we have $B(y_i,8\epsilon)\subset P_i^{\circ}$. Let $\nu$ be an inward unit normal vector of $F(P)$ and put 
\[
z_i:=y_i-2\epsilon \nu. 
\]We have $z_i\in B(y_i, 3\epsilon)\subset P_i^{\circ}$ and it converges to 
$z:=x-2\epsilon\nu$. On the other hand one can see that $B(z, \epsilon)\subset P^c$ and $\displaystyle z\in \lim_{i\to\infty}P_i=P$. It contradicts to $d^H(P_i,P)\to 0 \ (i\to\infty)$.  
\end{proof}

\begin{cor}\label{lowdimconv}
As in the same setting in Proposition~\ref{FtoF} for any $k=0,1,\ldots, n-1$ and a point $x\in F_{k}(P)$ there exists a sequence $\{x_i\in F_{k}(P_i)\}_i$ such that $x_i  \to x \ (i\to\infty)$.   
\end{cor}
\begin{proof}
For any $x\in F_{n-2}(P)$ let $F(P)$ be a facet of $P$ which contains $x\in F_{n-2}(P)$. By Proposition~\ref{FtoF} $F(P)$ can be described as a limit of a union of facets of $F(P_i)$. The proof of \ref{FtoF} shows that $F_{n-2}(P)$ can be described as a limit of $(n-2)$-dimensional faces of $F(P_i)$. One can prove the claim in an inductive way. 
\end{proof}

\begin{cor}\label{cor:lscfacets}
As in the same setting in Proposition~\ref{FtoF} the number of $k$-dimensional faces is lower semi-continuous for any $k$:  
\[
{\#}\{F^{(r)}_k(P)\}_r\leq\lim_{i\to\infty} ({\#}\{F^{(r)}_k(P_i)\}_r). 
\]
\end{cor}


\begin{cor}\label{ltol}
Consider the same setting in Proposition~\ref{FtoF}. For any facet $F^{(r)}(P)$, its normal vector $\nu^{(r)}$ and a scalar $\lambda^{(r)}$ there exists a sequence of facet $F^{(r_i)}(P_{i})$ such that the corresponding defining affine functions converges to that of $F(P)$, i.e., $l_{i}^{(r_i)}\to l^{(r)} \ (i\to\infty)$.
\end{cor}
\begin{proof}
By Proposition~\ref{FtoF}, for any facet $F^{(r)}(P)$ of $P$,  one can take a sequence of facets $\{F^{(r_i)}(P_i)\}_i$ of $P_i$ which converges to $F^{(r)}(P)$. We may assume that the sequence of unit normal vectors of $F^{(r_i)}(P_i)$ converges to that of $F^{(r)}(P)$. It implies that the corresponding defining affine functions $l_i^{(r_i)}$ converge to $l^{(r)}$. 
\end{proof}


We say a sequence of $k$-dimensional faces $\{F_k(P_i)\}_i$  of a sequence $\{P_i\}_i$ in $\calD_n$ {\it converges essentially to a $k$-dimensional face $F_k(P)$ of $P\in\calD_n$} if 
\[
\lim_{i\to\infty}\calH^k(F_k(P_i))>0
\] and 
\[
\lim_{i\to\infty} d^H(F_k(P_i), F)=0
\]for a closed subset $F$ of $F_k(P)$. 


Next we consider the $2$-dimensional case $\calD_2$. 

\begin{thm}\label{primtoprim}
For a sequence $\{P_i\}_i\subset \calD_2$ suppose that $d^H(P_i, P)\to 0 \ (i\to \infty)$ for some $P\in\calD_2$. 
For each facet $F^{(r)}(P)$ of $P$ and its primitive normal vector $\nu^{(r)}$, there exists a sequence of primitive normal vectors $\{\nu_i^{(r_i)}\}_i$ of $F^{(r_i)}(P_i)$ such that $\nu_i^{(r_i)}\to\nu^{(r)} \ (i\to \infty)$.
\end{thm}
\begin{proof}
By Corollary~\ref{cor:lscfacets} and the semi-continuity of the Hausdorff measure in the non-collapsing limit we may assume that for each facet (=edge)  $F^{(r)}(P)$ there exists a sequence $\{F^{(r_i)}(P_i)\}_i$ of facets of $\{P_i\}_i$ which converges essentially to $F^{(r)}(P)$. 

 We rearrange the indices so that $r=r_i=1$ for all $i$. Moreover we may assume that the facets are numbered in a counterclockwise way. Note that by the smoothness condition the determinant of the $2\times 2$ matrix consisting of any adjacent primitive normal vectors is $\pm 1$. 

Since $\{F^{(1)}(P_i)\}_i$ converges essentially to $F^{(1)}(P)$ the sequence of inward unit normal vectors converges : 
\[
\frac{\nu^{(1)}_i}{\|\nu^{(1)}_i\|}\to \frac{\nu^{(1)}}{\|\nu^{(1)}\|} \quad (i\to \infty).
\] 
Since $\{\nu^{(1)}_i\}_i$ is a sequence of integral vectors it suffices to show that $\{\|\nu^{(1)}_i\|\}_i$ is a bounded sequence. 

Suppose that $\{\|\nu^{(1)}_i\|\}_i$ is unbounded. In this case note that 
\[
\left|\det\left(\frac{\nu^{(1)}_i}{\|\nu^{(1)}_i\|}, \frac{\nu^{(2)}_i}{\|\nu^{(2)}_i\|}\right)\right|=\frac{1}{\|\nu_i^{(1)}\|\|\nu_i^{(2)}\|}|\det(\nu_i^{(1)}, \nu_i^{(2)})|\leq \frac{1}{\|\nu_i^{(1)}\|}\to 0 \quad (i\to \infty). 
\]
It implies that the facets $F^{(1)}(P_i)$ and $F^{(2)}(P_i)$ tends to be parallel as $i\to \infty$. 
The same situation holds for any pair of adjacent facets of $P_i$ in which at least one of the sequence of primitive normal vectors is unbounded. 
Now since $\{P_i\}_i$ converges to a simple convex polytope $P$, there exist at least two facets $\{F^{(r_i)}(P_i)\}_i$ and $\{F^{(r'_i)}(P_i)\}_i$ with $r_i \leq r_i'$ such  that  they converge essentially to some facets of $P$ which are adjacent to $F^{(1)}(P)$.
If $\{\nu_i^{(r)}\}_i$ are unbounded for $r=1,2,\ldots, r_i-1$, then the above argument of determinant shows that all facets $\{F^{(r)}(P_i)\}_i$ tend to be parallel to each other. It implies that $F^{(1)}(P_i)$ and $F^{(r_i)}(P_i)$ tend to be parallel each other, and it is a contradiction.  It implies that there exists $r\in\{2,\ldots, r_i-1\}$ such that $\nu^{(r)}_i$ is bounded. 
The same argument implies that there exists $r'\in\{r_i', \ldots, N_i\}$ such that $\nu^{(r'_i)}_i$ is bounded, where $N_i$ is the number of facets of $P_i$. 

Since $\{F^{(r)}(P_i)\}_i$ and $\{F^{(r')}(P_i)\}_i$ tend to be parallel to $F^{(1)}(P)$ and the bounded primitive normal vectors $\{\nu^{(r)}_i\}_i$ and $\{\nu_i^{(r')}\}_i$ are integral vectors then we may assume that $\nu^{(r_i)}_i=\nu_i^{(1)}=\nu^{(r_i')}_i$ for any sufficiently large $i$ (by taking a subsequence of the subsequence). It is a contradiction because such a situation cannot be realized in a convex polytope $P_i$. In particular $\{\|\nu^{(1)}_i\|\}_i$ is bounded, and it completes the proof of the theorem. 
\end{proof}

\begin{Rem}
In Theorem~\ref{primtoprim} the boundedness of each primitive normal vector $\{\nu_i^{(r)}\}_i$ 
implies that it contains a constant subsequence. 
\end{Rem}

By the same argument we have the following convergence in the higher dimensional non-degenerate case.  
\begin{thm}\label{thm:primtoprimgen}
For a sequence $\{P_i\}_i\subset \calD_n$ suppose that $d^H(P_i, P)\to 0 \ (i\to \infty)$ for some $P\in\calD_n$ and 
$\displaystyle{\#}\{F^{(r)}(P)\}_r=\lim_{i\to\infty} ({\#}\{F^{(r)}(P_i)\}_r)$.
For each facet $F^{(r)}(P)$ of $P$ and  its primitive normal vector $\nu^{(r)}$, there exists a sequence of primitive normal vectors $\{\nu_i^{(r_i)}\}_i$ of $F^{(r_i)}(P_i)$ such that $\nu_i^{(r_i)}\to\nu^{(r)} \ (i\to \infty)$.
\end{thm}

\begin{proof}
As in the proof of Theorem~\ref{primtoprim} we  can take a sequence of primitive normal vectors $\{\nu^{(1)}_i\}_i$ 
of $\{F^{(1)}(P_i)\}_i$, 
and it suffices to show that $\{\|\nu^{(1)}_i\|\}_i$ is bounded. 
Suppose that $\{\|\nu^{(1)}_i\|\}_i$ is unbounded. 
Consider a vertex of $F^{(1)}(P_i)$ and facets around it. We may assume that they are numbered as 
$r=2,3,\cdots, n$. 
Then for their primitive normal vectors we have 
\[
\left|\det\left(\frac{\nu^{(1)}_i}{\|\nu^{(1)}_i\|}, \frac{\nu^{(2)}_i}{\|\nu^{(2)}_i\|}, \cdots ,
\frac{\nu^{(n)}_i}{\|\nu^{(n)}_i\|}\right)\right| \leq \frac{1}{\|\nu_i^{(1)}\|}\to 0 \quad (i\to \infty). 
\]
It contradicts to our assumption $\displaystyle{\#}\{F^{(r)}(P)\}_r=\lim_{i\to\infty} ({\#}\{F^{(r)}(P_i)\}_r)$.
\end{proof}

\subsection{From convergence of polytope to convergence of Guillemin metric}

We first give the definition of equivariant (measured) Gromov-Hausdorff convergence 
as a special case of \cite[Definition~1-3]{Fukaya}.

\begin{defn}\label{def:eqGH}
Let $X=(X,d)$ be a compact metric space and $\{X_i=(X_i,d_i)\}_i$ be a sequence of compact metric spaces. 
Suppose that there exists a group $G$ which acts on $X$ and each $X_i$ 
in an effective and isometric way. 
Then $\{X_i\}_i$ converges to $X$ in the $G$-{\it equivariant Gromov-Hausdorff topology} if 
there exist sequences of maps $\{f_i:X_i\to X\}_i$, group automorphisms  
$\{\rho_i:G\to G\}_i$ and positive numbers $\{\epsilon_i\}_i$ 
such that the following conditions hold for any $i$ large enough. 
\begin{enumerate}
\item $\epsilon_i \to 0$ as $i\to \infty$. 
\item $|d_{i}(x,y)-d(f_i(x), f_i(y))|<\epsilon_i$  for all $x,y\in X_i$. 
\item For any $p\in X$ there exists $x\in X_i$ such that $d(p,f_i(x))<\epsilon_i$. 
\item $d(f_i(gx), \rho_i(g)f_i(x))<\epsilon_i$ for all $x\in X_i$ and $g\in G$. 
\end{enumerate}
This situation will be denoted by $X_i\xrightarrow{G\text{-eqGH}}X$ (or $X_i \to X$ for simplicity) and $f_i$ are called 
{\it approximation maps}. 

Moreover if $X$ (resp. $\{X_i\}_i$) is equipped with a $G$-invariant measure $m$ (resp. $m_i$)
in such a way that $(X,m)$ (resp. $(X_i, m_i)$) is a metric measure space and the push forward measure 
$(f_i)_{*}m_i$ converges to $m$ weakly, then we say 
$\{(X_i, m_i)\}_i$ converges to $(X, m)$ in the $G$-equivariant {\it measured} Gromov-Hausdorff topology 
and we will denote $X_i\xrightarrow{G\text{-eqmGH}}X$. 

When $X$ (resp. $X_i$) is a Riemannian manifold, we consider its Riemannian distance. 
\end{defn}

The above conditions (2), (3) and (4) mean that 
the approximation map $f_i$ is {\it almost isometric, almost surjective} and {\it almost equivariant}.

As a corollary of Theorem~\ref{thm:primtoprimgen} we have the following convergence theorem of symplectic toric manifolds. 
We emphasize that we do not put any assumptions on curvatures in our theorem below. 
\begin{thm}\label{thm:convGH}
For a sequence $\{P_i\}_i\subset \calD_n$ suppose that $d^H(P_i, P)\to 0 \ (i\to \infty)$ for $P\in\calD_n$ 
and $\displaystyle{\#}\{F^{(r)}(P)\}_r=\lim_{i\to\infty} ({\#}\{F^{(r)}(P_i)\}_r)$.  
Then there exists a subsequence of $\{M_{P_i}\}_i$ which converges to $M_P$ in the $T$-equivariant Gromov-Hausdorff topology. 
\end{thm}

\begin{proof}
We use the same notations as in Section~\ref{subsec:Delcon} with suffix $i$. 
We may assume $N={\#}\{F^{(r)}(P)\}_r={\#}\{F^{(r)}(P_i)\}_r=N_i$. 
The proof of Theorem~\ref{thm:primtoprimgen} implies that ${\frakh}_i={\frakh}$ and $H_i=H$ for $i\gg 0$. 
Moreover as a corollary of Theorem~\ref{thm:primtoprimgen} we have 
$\lambda_{i}^{(r)}\to \lambda^{(r)}$ ($i\to\infty$) for the constants of the defining equations of $P_i$ 
(after renumbering the facets). 
As a consequence $(\iota_i^*\circ\tilde\mu_i)^{-1}(0)$ converges to $(\iota^*\circ\tilde\mu)^{-1}(0)$ 
in the equivariant Gromov-Hausdorff topology\footnote{In fact this convergence is nothing other than the 
Hausdorff convergence of a sequence of compact subsets in $\bbR^N$. }. 
Then $\{M_{P_i}=(\iota_i^*\circ\tilde\mu_i)^{-1}(0)/H_i\}_i$ converges to 
$M_P=(\iota^*\circ\tilde\mu)^{-1}(0)/H$ in the Gromov-Hausdorff topology by \cite[Theorem2-1]{Fukaya}. 
Moreover the identifications $H_i=H$ induce identifications $T_i^n=T^N/H_i=T^N/H=T^n$, which 
makes the above convergence into the $T$-equivariant Gromov-Hausdorff topology.  
\end{proof}

\begin{cor}\label{cor:stability}
Under the same assumptions in Theorem~\ref{thm:convGH}, take a subsequence in 
$\{M_{P_i}\}_i$ which converges to $M_P$.  Then $M_{P_i}$ are $T$-equivariantly diffeomorphic to $M_P$ for $i\gg 0$.
\end{cor}
\begin{proof}
By Theorem~\ref{primtoprim} we may assume that $\nu_i^{(r)}=\nu^{(r)}$ for $i\gg 0$. 
On the other hand each $M_{P_i}$ is $T$-equivariantly diffeomorphic to the toric variety associated with 
the fan $\Sigma_{P_i}$. 
Note that $\Sigma_{P_i}$ is determined by the normal vectors $\{\nu_i^{(r)}\}_r$ and 
it does not depend on $\{\lambda_i^{(r)}\}_r$  (See \cite{BuchstaberPanov} for example).  
It implies the claim. 
\end{proof}

\begin{Rem}\label{ex:nonconv}
It can not be expected that a convergence as in Theorem~\ref{thm:convGH} occurs in general. 
%
Consider a sequence of Delzant pentagon $\{P_i\}_i$ as in Figure~\ref{pict:pentagon}, 
which converges to a rectangle $P$ defined by 5 inequalities. 
\begin{figure}[h]
\includegraphics[scale=0.5]{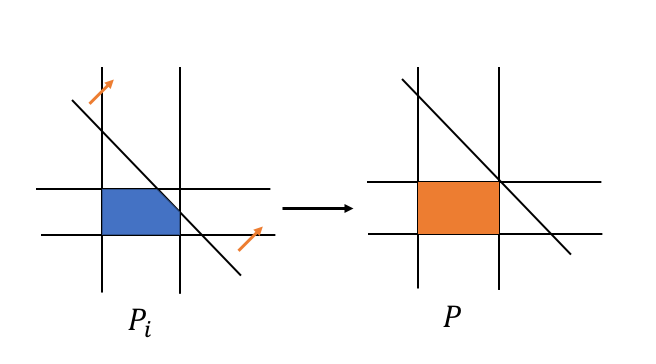}
\caption{A sequence of pentagons which converges to a rectangle} \label{pict:pentagon}
\end{figure}
It is known that the symplectic toric manifolds correspond to each pentagon $P_i$ are (diffeomorphic to) 
a 1 point blow-up of $\bbC P^1\times \bbC P^1$.  
The limiting process to $P$ corresponds to shrink the exceptional divisor in $M_{P_i}$. 
Their limit as symplectic quotient is defined by 5 inequalities, and it carries a Riemannian metric which is not isometric to the Guillemin metric. 
On the other hand in our setting $M_P$  is $\bbC P^1\times \bbC P^1$ equipped with the Guillemin metric. 
%
To deal with these subtle phenomena we have to consider finer structures on $\calD_n$ or $\widetilde\calD_n$ 
and incorporate {\it potential functions}. 
We will discuss such formulation in a subsequent paper. 
\end{Rem}

\subsection{From convergence of Guillemin metrics to convergence of polytopes}
Now let us discuss the convergence of the opposite direction.

Hereafter for each $P\in\calD_n$ we denote the symplectic toric manifold equipped with the Guillemin 
metric by $M_P=(M_P, \omega_P)$, and 
we use the Liouville volume form ${\vol}_{M_P}:=\frac{(\omega_P)^{\wedge n}}{n!}$ on the symplectic toric manifold $M_P$. 
In this way we think $M_P$ as a metric measure space. 

\begin{thm}\label{thm:GHHvertex}
Let $\{P_i\}_i$  be a sequence in $\calD_n$.  Suppose that a sequence of symplectic toric manifolds  
$\{M_{P_i}\}_i$ converges to $M_P$ for some $P\in\calD_n$ 
in the $T$-equivariant measured Gromov-Hausdorff topology.  
Let $\{f_i:M_{P_i}\to M\}_i$ be a sequence of approximation maps of the convergence. 
If $\{P_i\}_i$ are contained in a sufficiently large ball in $\bbR^n$, 
then we have 
\[
\lim_{i\to\infty}f_i(M_{P_i}^T)=M_P^T, 
\]where $M_{P_i}^T$ and $M_P^T$ are the fixed point sets of $T$-actions. 
In particular we have 
\[
\lim_{i\to\infty}\chi(M_{P_i})\geq \chi(M_P), 
\]where $\chi(\cdot)$ denotes the Euler characteristic. 
\end{thm}

\begin{proof}
For simplicity we denote $M_i:=M_{P_i}$ and $M:=M_P$. 
We first show that $\displaystyle\lim_{i\to\infty}f_i(x_i)\in M^T$ for any sequence $\{x_i\in M_i^T\}_i$.  
Suppose that there exists $\delta>0$ such that $f_i(x_i)\notin B(M^T, {\delta})$ for 
infinitely many $i$.  
For $\epsilon > 0$, we define $\delta_\epsilon$ as the minimal $\delta' > 0$ such that if $y \not\in B(M^T, \delta')$, then $\Diam(T\cdot y) \ge \epsilon$. 
Note that since $M$ is compact such $\delta_\epsilon>0$ exists and $\delta_\epsilon\to 0$ as $\epsilon\to 0$. 
Since $f_i$ is almost $T$-equivariant we have 
\[
\epsilon_i>d(\rho_i(t)f_i(x_i), f_i(tx_i))=d(\rho_i(t)f_i(x_i), f_i(x_i))
\]for all $t\in T$, where $\{\epsilon_i\}_i$ is a sequence of positive numbers as in Definition~\ref{def:eqGH} and  
$d$ is the Riemannian distance of $M$. 
It implies that ${\Diam}(T\cdot f_i(x_i))<2\epsilon_i\to 0$ as $i\to \infty$. 
If we take $i$ large enough so that $\delta_{\epsilon_i}<\delta$, then we have  
$f_i(x_i)\in B(M^T, {\delta_{\epsilon_i}})$. It contradicts to $f_i(x_i)\notin B(M^T, {\delta})$. 

Next we show that for any $\delta>0$ there exists $i_0\in\bbN$ such that 
\[
f_i^{-1}(M^T)\subset B(M_i^T, {\delta})
\]holds for all $i>i_0$. If not then there exists $\delta>0$ such that we can take $x_i\in f_i^{-1}(M^T)$ and 
$x_i\notin B(M_i^T, {\delta})$ for infinitely many $i$. 
Since $f_i$ is almost isometry and almost $T$-equivariant we have 
\begin{eqnarray*}
d_i(tx_i, x_i)&<&d(f_i(tx_i), f_i(x_i))+\epsilon_i \\ 
&<&d(f_i(tx_i), tf_i(x_i))+\epsilon_i \\
&<&2\epsilon_i
\end{eqnarray*} for all $t\in T$, where $d_i$ is the Riemannian distance of $M_i$. 
It implies ${\Diam}(T\cdot x_i)<4\epsilon_i$. 
On the other hand it is known that each $T\cdot x_i$ is a flat torus, and hence, 
${\Diam}(T\cdot x_i)\to 0$ $(i\to\infty)$ implies 
${\rm Vol}(T\cdot x_i)\to 0 \ (i\to \infty)$,  where ${\rm Vol}$ is the Riemannian volume with respect to the 
induced Riemannian metric. 
Now consider a compact subset $P_i':=\mu_i(M_i\setminus B(M_i^T, \delta))$ of $P_i$. 
Since $\{M_i\}_i$ converges to $M$ in the measured Gromov-Hausdorff topology 
$\{{\rm Vol}(M_i)\}_i$ converges to ${\rm Vol}(M)$. 
Duistermaat-Heckman's theorem implies that the Euclidean volumes of $\{P_i\}_i$ converge to that of $P$. 
In particular they are bounded below by a positive constant. 
Moreover since we assume that $\{P_i\}_i$ are contained in a ball, the sequence of convex polytopes $\{P_i\}_i$ 
converges to some convex body $Q$ in the Hausdorff distance. As in the same way $\{P_i'\}_i$ converges to 
some compact subset $Q'$ of $Q$. 
Let $Q^{(0)}$ be the limit point set of $\mu_i(M_i^T)=P_i^{(0)}$. 
Then we have $Q^{(0)}\cap Q'=\emptyset$.  
When we take $\delta'>0$ small enough so that ${\rm dist}(Q^{(0)},Q' )>2\delta'$ we have 
${\rm dist}(P^{(0)}_i, P'_i)>\delta'$. 
The formula of volumes of the orbits in \cite{IriyehOno} implies
\footnote{Strictly speaking the formula in \cite{IriyehOno} can be applied when $\mu_i(x_i)$ is in the 
interior part of $P_i$. So the above argument shows that $\{x_i\}_i$ cannot be taken in such an interior part. 
As the next step we assume that $\{x_i\}_i$ sits in the inverse image of the interior part of codimension one face, 
and we deduce the contradiction. We proceed the same step for higher codimension face. } 
that 
\[
\liminf_{i\to \infty}{\rm Vol}(T\cdot x_i)>0. 
\]It contradicts to $\displaystyle\lim_{i\to \infty}{\rm Vol}(T\cdot x_i)=0$. 

The inequality 
\[
\lim_{i\to\infty}\chi(M_{P_i})\geq \chi(M_P), 
\]
follows from the fact that the Euler characteristic of symplectic toric manifold 
is equal to the number of fixed points. 
\end{proof}

Hereafter we discuss the convergence of polytopes under the same assumption in Theorem~\ref{thm:GHHvertex}. 
We first take and fix a section $S_i:P_i\to M_{P_i}$ of the moment map $\mu_i:M_{P_i}\to P_i$ for each $i$. 
Note that each $S_i$ is neither smooth nor continuous but only measurable in general. 
Let $\{f_i:M_{P_i}\to M_P\}_i$ be a sequence of approximation maps. 
It is known that we may assume that $f_i$ is a Borel measurable map. 
For each $i$ we define $F_i:P_i\to P$ by the composition $F_i:=\mu\circ f_i \circ S_i$.

\begin{thm}\label{thm:GHH}
Under the same assumptions in Theorem~\ref{thm:GHHvertex}  there exists a subsequence of 
$\{\overline{F_i(P_i)}\}_i$ which converges to $P$ in $d^H$ topology. 
\end{thm}

 To show Theorem~\ref{thm:GHH} we prepare two lemmas. 
 
\begin{lem}\label{lem:disapprox}
 Consider the same setting as in Theorem~\ref{thm:GHH}. 
 For any $\varphi\in C_b(\bbR^n)$ there exists a sequence of measurable maps 
 $\{\varphi_i: P_i\to \bbR\}_i$ such that 
 \[
 \lim_{i\to \infty}\int_{P_i}\varphi_i~d\calL^n=\int_{P}\varphi ~d\calL^n. 
 \]
 \end{lem}

\begin{proof} 
Let $\mu_i:M_{P_i}\to P_i\subset \bbR^n$ and $\mu:M\to P\subset \bbR^n$ be the
 moment maps. By Duistermaat-Heckman's theorem we have 
$(\mu_i)_*({\vol}_{M_{P_i}})=\calL^n|_{P_i}$. 

For $\varphi\in C_b(\bbR^n)$ we define $\tilde\varphi\in C(M_P)$ by $\tilde\varphi:=\varphi\circ \mu$. 
Let $\{f_i\}_i$ be a family of approximation maps for $M_{P_i}\xrightarrow{T\text{-eqmGH}}M_P$.  
We define a sequence of measurable functions $\{\tilde\varphi_i:M_i\to \bbR\}_i$ by 
$\tilde\varphi_i:=\tilde\varphi\circ f_i$. 
Let $\{({\vol}_{M_{P_i}})_y\}_{y\in P_i}$ (resp. $\{({\vol}_{M_{P}})_y\}_{y\in P}$) be a disintegration 
(See Appendix~\ref{App.B}) for $\mu_i:M_{P_i}\to P_i$ (resp. $\mu:M_P\to P$) and define 
a sequence of measurable functions $\{\varphi_i:P_i\to\bbR\}_i$ by 
\begin{equation}\label{def:varphi_i}
\varphi_i(y):=\int_{M_{P_i}}\tilde\varphi_i(x)({\vol}_{M_{P_i}})_y(dx).
\end{equation}
Since $(f_i)_*({\vol}_{M_{P_i}})$ converges to ${\vol}_{M_P}$ weakly we have 
\begin{eqnarray*}
\int_{P_i}\varphi_i(y) \calL^n(dy)&=&\int_{P_i}\left(\int_{M_{P_i}}\tilde\varphi_i(x)({\vol}_{M_{P_i}})_y(dx)\right)\calL^n(dy)\\
&=&\int_{M_{P_i}}\tilde\varphi_i(x){\vol}_{M_{P_i}}(dx)
=\int_{M_{P_i}}\tilde\varphi(f_i(x)){\vol}_{M_{P_i}}(dx) \\
&\xrightarrow[i\to\infty]{}&
\int_{M_P}\tilde\varphi(x){\vol}_{M_P}(dx)\\
&=&\int_{P}\left(\int_{M_P}\tilde\varphi(x)({\vol}_{M_{P}})_y(dx)\right)\calL^n(dy) \\ 
&=&\int_{P}\left(\int_{\mu^{-1}(y)}\varphi(\mu(x))({\vol}_{M_{P}})_y(dx)\right)\calL^n(dy) \\ 
&=& \int_{P}\left(\int_{\mu^{-1}(y)}\varphi(y)({\vol}_{M_{P}})_y(dx)\right)\calL^n(dy) \\ 
&=&\int_P\varphi(y)\calL^n(dy). 
\end{eqnarray*}
\end{proof}

\begin{lem}\label{lem:approxconstr}
As in the same setting in Theorem~\ref{thm:GHH} we have
\[
\lim_{i\to \infty}\frac{1}{|P_i|}\int_{P_i}\varphi_i d\calL^n=\lim_{i\to\infty}\frac{1}{|P_i|}\int_{P_i}\varphi\circ F_i d\calL^n. 
\] 
for any $\varphi\in C_b(\bbR^n)$, where $\varphi_i$ are as in Lemma~\ref{lem:disapprox}. 
\end{lem}

\begin{proof}
Let $\{\rho_i:T^n\to T^n\}_i$ be a sequence of automorphisms as in Definition~\ref{def:eqGH} for
$M_{P_i}\xrightarrow{{\rm eq-m}GH}M_P$. 
Fix $\eta>0$ and $\varphi\in C_b(\R^n)$. 
For any $y\in P_i$ we have 
\begin{equation}\label{estimate1}
|\varphi_i(y)-\varphi(F_i(y))|
\leq \int_{\mu_i^{-1}(y)}|\varphi(\mu(f_i(x)))-\varphi(\mu(f_i(S_i(y))))|({\vol}_{M_{P_i}})_y(dx). 
\end{equation} Since for any $x\in \mu_i^{-1}(y)$ there exists $t_x\in T$ such that $x=t_x\cdot S_i(y)$ we have 
\begin{eqnarray*} 
\|\mu(f_i(x))-\mu(f_i(S_i(y)))\|&
=&\|\mu(f_i(t_x\cdot S_i(y)))-\mu(f_i(S_i(y)))\| \\ 
&=& \|\mu(f_i(t_x\cdot S_i(y)))-\mu(\rho_i(t_x)\cdot f_i(S_i(y)))\|. 
\end{eqnarray*}
On the other hand since $\varphi$ and $\mu$ are uniformly continuous 
and $\{M_{P_i}\}_i$ converges to $M_P$ in the $T$-equivariant Gromov-Hausdorff topology 
there exists $i_0\in\bbN$ such that if $i>i_0$, then 
\[
|\varphi(\mu(f_i(x)))-\varphi(\mu(f_i(S_i(y))))|=|\varphi(\mu(f_i(x)))-\varphi(\mu(\rho_i(t_x)\cdot f_i(S_i(y))))|<\eta. 
\]
In particular we have 
\[
|\varphi_i(y)-\varphi(F_i(y))|<\eta
\] in (\ref{estimate1}), and hence, 
\[
\frac{1}{|P_i|}\left|\int_{P_i}(\varphi_i(y)-\varphi(F_i(y)))\calL^n(dy)\right|<\eta. 
\]Note that our assumption $M_{P_i}\xrightarrow{T\text{-eqmGH}}M_P$ and 
Duistermaat-Heckman's theorem imply $|P_i|={\vol}_{M_{P_i}}(M_{P_i})\to  |P|={\vol}_{M_{P}}(M_{P})$. 
Since $\eta>0$ is arbitrary 
the limit of $\displaystyle\frac{1}{|P_i|}\int_{P_i}\varphi_i(y)\calL^n(dy)$ exists and we have the required equality 
\[
\lim_{i\to\infty}\frac{1}{|P_i|}\int_{P_i}\varphi_i(y)\calL^n(dy)=
\lim_{i\to\infty}\frac{1}{|P_i|}\int_{P_i}\varphi(F_i(y))\calL^n(dy). 
\]
\end{proof}

\begin{proof}[Proof of Theorem~\ref{thm:GHH}]
Let $\varphi\in C_b(\bbR^n)$. By Lemma~\ref{lem:disapprox} and Lemma~\ref{lem:approxconstr} we have a sequence of measurable maps $\{F_i:P_i\to P\}_i$ 
and measurable functions $\{\varphi_i:P_i\to \bbR\}_i$ such that 
\[
\lim_{i\to\infty}\int_P\varphi(y)(F_i)_*(\calL^n)(dy)=
\lim_{i\to\infty}\int_{P_i}\varphi_i(y)\calL^n(dy)
=\int_P\varphi(y)\calL^n(dy). 
\]
Note that we have $|P_i|\to |P| \ (i\to\infty)$ 
under our assumption, measured Gromov-Hausdorff convergence, and Duistermaat-Heckman's theorem. 
This equality implies that the sequence of probability measures 
$\{(F_i)_{*}m_{P_i}\}_i$ converges weakly to $m_{P}$. 
Since 
$F_i(P_i)\subset P$ we have 
\[
\lim_{R\to \infty}\limsup_{i\to\infty}\int_{\bbR^n\setminus B(0, R)}\|x\|^2(F_i)_{*}m_{P_i}(dx)=0. 
\]
It implies that $W_2((F_i)_{*}m_{P_i}, m_P)\to 0 \ (i\to \infty)$ by (1) and  (2) in Theorem~\ref{thm:W2convweakconv}, 
and hence, ${\supp}((F_i)_{*}(m_{P_i}))=\overline{F_i(P_i)}$ converges to $P$ as $i\to \infty$.  
\end{proof}

\begin{Rem}
Regarding Theorem~\ref{thm:GHHvertex} and Theorem~\ref{thm:GHH} let us mention some comments. 
It is natural to consider the following two problems;  removing the assumption on uniformly boundedness of $\{P_i\}_i$ and 
getting a convergence of $\{P_i\}_i$ to $P$ in the Gromov-Hausdorff or $d^H$-topology. 
One can see that these are not true in the literal sense because of the ambiguity of the affine transformation groups $G_n$. 
We could address these problems in terms of the moduli space. 
Namely one may hope that if $\{M_{P_i}\}_i$ converges to $M_P$ in the $T$-equivariant measured Gromov-Hausdorff topology, 
then there exists a sequence $\{\varphi_i\}_i$ in $G_n$ such that $\{\varphi_i(P_i)\}_i$ converges to $P$ 
in the Gromov-Hausdorff  or $d^H$-topology. 
It would be useful to consider minimum variance elements explained in Remark~\ref{rem:minvar}. 

\end{Rem}

\appendix
 \section{Preliminaries on probability measures and $L^2$-Wasserstein distance}\label{App.A}
 \label{Preliminaries}
In this appendix we summarize several facts on probability measures and $L^2$-Wasserstein distance. 
For more details consult \cite{Villanibook} for example. 

Let $\scrP(\bbR^n)$ be the set of all complete Borel probability measures on $\bbR^n$. 
Consider the subset of $\scrP(\bbR^n)$ consisting of measures with finite quadratic moment, 
\begin{align}
  \scrP_2(\bbR^n)&:=\left\{m\in\scrP(\R^n) \ \middle| \ \exists o\in\R^n, \ \int_{\R^n}\|x-o\|^2m(dx)<\infty\right\}\notag. 
 \end{align} 



\subsection{Weak convergence and Prokhorov's theorem}
\begin{defn}
A sequence $\{m_i\}_i$ in $\scrP(\R^n)$ {\it converges weakly to $m\in\scrP(\R^n)$} 
\[
\lim_{i\to\infty}\int_{\R^n}f(x)m_i(dx)=\int_{\R^n}f(x)m(dx)
\]for any bounded continuous function $f$ on $\R^n$. 
\end{defn}

\begin{thm}\label{thm:propertyweak}
For a sequence $\{m_i\}_i$ in $\scrP(\R^n)$ and $m\in\scrP(\R^n)$ the followings are equivalent. 
\begin{enumerate}
  \item $\{m_i\}_i$ converges weakly to $m$.
  \item For any open subset $U$ in $\R^n$ we have $\displaystyle\liminf_{i\to\infty} m_i(U)\geq m(U)$. 
  \item For any  closed subset $C$ in $\R^n$ we have $\displaystyle\limsup_{i\to\infty} m_i(C)\leq m(C)$. 
  \item For any Borel subset $A$ in $\R^n$ with $m(\overline{A}\setminus A^{\circ})=0$ we have 
  $\displaystyle\lim_{i\to\infty} m_i(A)= m(A)$. 
 \end{enumerate}
\end{thm}

\begin{thm}[Prokhorov's theorem]\label{thm:Prokhrov}
A subset $\calK\subset\calP(\R^n)$ is relatively compact with respect to the weak convergence topology if and only if for all $\epsilon>0$ there exists a compact subset $K\subset\R^n$  such that \footnote{A subset $\calK\subset\calP(\R^n)$ with this property is often called {\it tight}. } 
\begin{align}
  \sup_{ m\in\calK} m(\R^n\setminus K)<\epsilon\notag. 
 \end{align}
\end{thm}

For a weak convergent sequence of probability measure the following is well-known. 
See \cite{AmbrosioGigliSavare} for example. 

\begin{thm}\label{thm:Kratow}
If $\{ m_i\}_i\subset\calP(\R^n)$ has a weak convergent limit $ m \in\calP(\R^n)$, then for any $x\in \supp(m)$ there exists $x_i\in\supp(m_i)$ such that  $x_i\rightarrow x$. 
\end{thm}

\subsection{$L^2$-Wasserstein distance of probability measures}
For $m, m'\in\scrP_2(\R^n)$ let $\Cpl(m, m')$ be the set of all couplings between $m$ and $m'$. Namely $\Cpl(m, m')$ is the set of measures $\xi\in\scrP(\R^n\times\R^n)$ such that for any Borel subset $A$ of $\R^n$ it satisfies 
\begin{align}
 \begin{cases}
  \xi(A\times \R^n)=m(A)\\
  \xi(\R^n\times A)=m'(A). 
 \end{cases}\notag
\end{align}
The $L^2$-Wasserstein distance between $m, m'\in\scrP_2(\R^m)$ is defined by 
\begin{align}
 W_2(m,m'):=\inf\left\{ \left(\int_{\R^n\times\R^n}\|x-y\|^2\xi(dx, dy)\right)^{1/2} \ \middle| \ \xi\in\Cpl(m,m')\right\}.\notag
\end{align} 
It is known that $W_2$ is a metric on $\scrP_2(\R^n)$ and $(\scrP_2(\R^n), W_2)$ is a complete separable metric space with the following properties.

\begin{thm}\label{thm:W2convweakconv}
For a sequence $\{m_i\}_i$ in $\scrP_2(\R^n)$ and $m\in\scrP_2(\R^n)$ the followings are equivalent.  
 \begin{enumerate}
  \item $W_2(m_i,m)\rightarrow 0$ \ ($i\to\infty$). 
  \item $\{m_i\}_i$ converges weakly to $m$ and 
  \begin{align}
   \lim_{R\rightarrow\infty}\limsup_{i\rightarrow \infty}\int_{\R^n\setminus B(o,R)}\|x-o\|^2m_i(dx)=0.\notag
  \end{align}
  \item For any continuous function $\varphi$ such that $\vert \varphi(x)\vert\leq C(1+\|x_0-x\|)^2$ for some $C>0$, $x_0\in\R^n$ the following holds. 
  \begin{align}
   \lim_{i\rightarrow\infty}\int_{\R^n}\varphi\,dm_i=\int_{\R^n}\varphi\,dm. \notag
  \end{align}

 \end{enumerate}
\end{thm}

Recall that if for $m, m'\in\scrP_2(\R^n)$ there exists a Borel measurable map $T:\R^n\to\R^n$ such that $T_*m=m'$ and $(\mathsf{id}\times T)_*m\in\Opt(m,m')$, then we say that the Monge problem for $m, m'$ admits a solution and $T$ is called a solution of the Monge problem. 

\begin{thm}
For $m, m'\in\scrP_2(\R^n)$ if $m\ll\calL^n$, then there is a solution of the Monge problem for $m$ and $m'$. The solution is unique in the following sense. For another solution $S:\R^n\to\R^n$ we have $m(\{T\neq S\})=0$. 
\end{thm}

\begin{cor}\label{cor:inprob}
For $m,m'\in\scrP_2(\R^n)$ with $m\ll\calL^n$ and a sequence $\{m'_i\}_i$ in $\scrP_2(\R^n)$ which 
converges weakly to $m'$, there exists a solution $T:\R^n\to\R^n$ of the Monge problem for $m$, $m'$ 
and a sequence $\{T_i\}_i$ of solutions of the Monge problem for $m$, $m_i'$ with 
\begin{align}
  m\left(\left\{x\in\R^n \ | \ \vert T_i(x)-T(x)\vert\geq \epsilon\right\}\right)\rightarrow 0 \ (i\to \infty). \notag
 \end{align}
\end{cor}

\section{Disintegration theorem}\label{App.B}
\label{Disintegration}
We use the following type of disintegration theorem. See \cite[Theorem~16.10.1]{Garlingbook} for example. 
\begin{thm}
Let $X$ and $Y$ be complete separable metric spaces. 
Let $m$ be a $\sigma$-finite Borel probability measure and $f:X\to Y$ a Borel measurable map. 
Suppose that the push forward $f_*m$ is a $\sigma$-finite measure on $Y$.  
Then there exists a family of probability measures $\{m_y\}_{y\in Y}$ on $X$ such that 
for each Borel subset $A$ the map 
\[
Y\ni y\mapsto m_y(A)\in[0,1] 
\]is Borel measurable and for each Borel measurable function $\varphi$ on $X$ we have 
\[
\int_X\varphi~dm =\int_Y\left(\int_X\varphi(x) m_y(dx)\right)f_*m(dy). 
\]Moreover we have 
\[
m_y(f^{-1}(y))=1 \quad (y\in Y \  (f_*m\text{-a.e}) ). 
\]
\end{thm}

The above family of measures $\{m_y\}_{y\in Y}$ is called a {\it disintegration for} $f:X\to Y$. 

\bibliography{refdelbeppufujimitsu}


\end{document}